\documentclass[]{amsart}
\usepackage{amsfonts}
\usepackage{graphicx}
\usepackage{amscd}
\usepackage{amsmath}
\usepackage{amssymb}
\usepackage{xcolor}
\usepackage{hyperref}

\makeatletter
\@namedef{subjclassname@2010}{%
  \textup{2010} Mathematics Subject Classification}
\makeatother

\usepackage{mathrsfs}
\usepackage{amsbsy}

\newtheorem{corollary}{\bf Corollary}

\newtheorem{lemma}{\bf Lemma}

\newtheorem{proposition}{\bf Proposition}
\newtheorem{remark}{Remark}

\newtheorem{theorem}{\bf Theorem}

\theoremstyle{definition}
\newtheorem{example}{\bf Example}

\numberwithin{equation}{section}

\makeatletter
\@namedef{subjclassname@2020}{%
  \textup{2020} Mathematics Subject Classification}
\makeatother

\makeatletter
\@namedef{subjclassname@2020}{%
  \textup{2020} Mathematics Subject Classification}
\makeatother

\begin{document}
	
	\title[Quasi-Einstein manifolds]{On quasi-Einstein manifolds \\with constant scalar curvature }
	\author[J. Costa, E. Ribeiro Jr and M. Santos]{Johnatan Costa, Ernani Ribeiro Jr and M\'arcio Santos}

\address[J. Costa]{Instituto Federal de Educa\c c\~ao, Ci\^encia e Tecnologia do Sert\~ao Pernambucano - IFSert\~aoPE, Campus Ouricuri, 56200-000 - Ouricuri - PE, Brazil.}\email{costajohnatan31@gmail.com}	

\address[E. Ribeiro Jr]{Universidade Federal do Cear\'a - UFC, Departamento  de Ma\-te\-m\'a\-ti\-ca, Campus do Pici, Av. Humberto Monte, Bloco 914, 60455-760, Fortaleza - CE, Brazil.}\email{ernani@mat.ufc.br}
	
\address[M. Santos]{Departamento de Matem\'atica, Universidade Federal da Para\'iba - UFPB, 58051-900, Jo\~ao Pessoa - PB, Brazil.}
\email{marcio.academico@ufpb.br}

\thanks{M. Santos was partially supported by CNPq/Brazil [Grant: 306524/2022-8]}
\thanks{E. Ribeiro was partially supported by CNPq/Brazil [Grant: 309663/2021-0 \& 403344/2021-2], CAPES/Brazil and FUNCAP/Brazil [Grant: ITR-0214-00116.01.00/23].}


\begin{abstract}
In this article, we study quasi-Einstein manifolds with cons\-tant scalar curvature. We provide a classification of compact and noncompact (possibly with boundary) $T$-flat quasi-Einstein manifolds with constant scalar curvature, where the $T$-tensor is directly related to the Cotton and Weyl tensors. Moreover, we construct new explicit examples of noncompact quasi-Einstein manifolds. In addition, we prove a complete classification of compact and noncompact (possibly with boundary) $3$-dimensional $m$-quasi-Einstein manifolds with constant scalar curvature. 
\end{abstract}
	
\date{\today}
	
	
\maketitle
	
\section{Introduction}\label{SecInt}

A  Riemannian manifold $(M^n,\,g)$, $n\geq 2,$ possibly with boundary $\partial M,$ is called an {\it $m$-quasi-Einstein manifold} if there exists a smooth potential function $u$ on $M^n$ satisfying the system
\begin{equation}
	\label{eq-qE}
	\left\{%
	\begin{array}{lll}
		\displaystyle \nabla^{2}u = \dfrac{u}{m}(Ric-\lambda g) & \hbox{in $M,$} \\
		\displaystyle u>0 & \hbox{on $int(M),$} \\
		\displaystyle u=0 & \hbox{on $\partial M,$} \\
	\end{array}%
	\right.
\end{equation} for some constants $\lambda$ and $0<m<\infty,$ see, for instance, \cite{CaseShuWei,HPW,HPW2}. We say that an $m$-quasi-Einstein manifold is {\it trivial} if the potential function $u$ is constant; otherwise, it is called {\it nontrivial}. It is worth noting that any $m$-quasi-Einstein manifold with $m>1$ admits a constant $\mu$ such that
\begin{eqnarray}
\label{eqmu}
	\mu=u\Delta u+(m-1)|\nabla u|^2+\lambda u^2.
\end{eqnarray} According to \cite{HPW2}, an $m$-quasi-Einstein manifold $(M^n,\,g,\,u,\lambda)$ is said to be {\it rigid} if it is Einstein or if its universal cover is isometric to a Riemannian product of Einstein manifolds. For more details, see, e.g., \cite{Besse,CaseShuWei,CasePJM,CMMM,compact,remarks,HPW,HPW2,Rimoldi,WangQEM1,WangQEM2} and the references
therein.

As pointed out by Besse \cite[p. 267]{Besse}, quasi-Einstein manifolds naturally arise as the base of warped product Einstein metrics. They also appear in the context of smooth metric measure spaces and diffusion operators, as studied by Bakry and
\'Emery \cite{bakry}. In the case $m=1,$  it is common to additionally assume $\Delta u=-\lambda u$ in order to recover the so-called {\it static spaces}, of particular interest in general relativity. Moreover, quasi-Einstein manifolds are also associated with the geometry of degenerate Killing horizons and horizon limits (see, e.g., \cite{BGKW,BGKW2,KWW,Wylie}). By setting $u=e^{-\frac{f}{m}}$ in (\ref{eq-qE}), one sees that when $\partial M=\emptyset,$ an $\infty$-quasi-Einstein manifold corresponds precisely to a gradient Ricci soliton \cite{CaoSurvey,Hamilton}. Given their significance, it is natural to seek classification results for quasi-Einstein manifolds.

In \cite{BRS,Bergery,Besse,CaseP} and \cite{WangQEM2}, the authors presented some examples of complete quasi-Einstein manifolds with $\lambda<0$ and arbitrary $\mu$, as well as examples of quasi-Einstein manifolds with $\lambda=0$ and $\mu>0.$ In contrast, Case \cite{CasePJM} proved that any complete quasi-Einstein manifold with $\lambda=0$ and $\mu\leq 0$ must be trivial. Furthermore, by the works of Qian \cite{Qian}, Kim-Kim \cite{kk} and He-Petersen-Wylie \cite{HPW2}, it is known that a complete $m$-quasi-Einstein manifold is compact if and only if $\lambda>0.$ For the case of compact with boundary manifolds, we refer to \cite[Theorem 4.1]{HPW2}. Therefore, any nontrivial (possibly with boundary) noncompact $m$-quasi-Einstein manifold has $\lambda\leq 0.$ For a comprehensive overview, we refer the reader to \cite{Barros,BRS,BRR,Besse,CaseShuWei,CMMM,CH,ChengRZhou,Fl-Gr2,HPW,HPW2,Ernani_Keti,Rimoldi,
WangQEM1,WangQEM2}.

In \cite{HPW2}, He, Petersen and Wylie showed that a nontrivial quasi-Einstein manifold $(M^n,\,g,\,u)$ with constant Ricci curvature must be isometric to either the standard hemisphere $\mathbb{S}_+^n,$ or $[0,\infty)\times F$ with the product metric $g=dt^2+g_{F},$ or the hyperbolic space $\mathbb{H}^n,$ or $[0,\infty)\times N$ equipped with the metric $g=dt^2+\sqrt{-k}\cosh^2(\sqrt{-k}t)g_{\mathbb{S}^{n-1}},$ or $\mathbb{R}\times F$ with the metric $g=dt^2+e^{2\sqrt{-k}t}g_{F};$ where $F$ is Ricci flat, $N$ is an Einstein metric with negative Ricci curvature and $k=\frac{\lambda}{m+n-1}.$

This result of He-Petersen-Wylie \cite[p. 160]{HPW2} leaves open the question of \textit{whether a complete classification of quasi-Einstein manifolds with constant scalar curvature can be established}. This line of investigation is also motivated by a recent classification of gradient shrinking Ricci solitons with constant scalar curvature by Cheng and Zhou \cite{ChengZhou}, and Fern\'andez-L\'opez and Garc\'ia-R\'io \cite{Fl-Gr}. In this context, by the work of Case-Shu-Wei \cite{CaseShuWei}, it is known that any compact quasi-Einstein manifold without boundary and with constant scalar curvature must be tri\-vial. Nevertheless, there exist known examples of nontrivial compact quasi-Einstein manifolds with boundary and constant scalar curvature, such as the standard hemisphere $\mathbb{S}_+^n$, the cylinder $I \times \mathbb{S}^{n-1}$ equipped with the product metric, and an example\footnote{This example was later adapted to produce a counterexample to the Cosmic No-Hair Conjecture in arbitrary dimensions $n\geq 4$, see \cite{CDPR}.} constructed by Di\'ogenes, Gadelha and Ribeiro \cite{remarks} on $\mathbb{S}^{p+1}_{+}\times \mathbb{S}^q$, $q>1,$ endowed with the metric $g=dr^2+\sin^2r g_{\mathbb{S}^p}+\frac{q-1}{p+m}g_{\mathbb{S}^q}$. 

It follows from the classical Reilly’s formula \cite[Theorem B]{Reilly} that the hemisphere $\mathbb{S}^2_{+}$ is the only nontrivial $2$-dimensional simply connected compact $m$-quasi-Einstein manifold with boundary and constant scalar curvature (see also \cite[pg. 271]{Besse}). Recently, inspired by ideas outlined by Cheng and Zhou \cite{ChengZhou} and Fern\'andez-Lop\'ez and Garc\'ia-R\'io \cite{Fl-Gr}, Costa, Ribeiro and Zhou \cite{costa-ribeiro-zhou} proved that a $3$-dimensional simply connected compact quasi-Einstein manifold with boundary and constant scalar curvature is isometric to either the standard hemisphere $\mathbb{S}^{3}_{+},$ or the cylinder $I\times \mathbb{S}^{2}$ with product metric. Moreover, they showed that a simply connected $4$-dimensional compact quasi-Einstein manifold with boundary and constant scalar curvature is isometric to one of the following: the standard hemisphere $\mathbb{S}^{4}_{+},$ the cylinder $I\times \mathbb{S}^{3}$ with product metric, or $\mathbb{S}^{2}_+\times\mathbb{S}^{2}$ also equipped with the product metric. However, the rigidity problem in dimensions $n\geq 5$ remains open.

In this article, following our very recent work \cite{costa-ribeiro-zhou}, we investigate the classification of noncompact (possibly with boundary) quasi-Einstein manifolds with constant scalar curvature. In this setting, as previously mentioned, we necessarily have $\lambda\leq 0.$ In particular, it is known from \cite[Theorem 1.8]{HPW2} and \cite{CaseShuWei} that if a quasi-Einstein manifold with $\lambda=0$ has constant scalar curvature, then it is either Ricci flat without boundary or $([0,\infty)\times F,\,g=dr^2+g_{F}),$ where $(F,\,g_{F})$ is Ricci flat and $u$ is a linear function. For dimension $n=2,$ it follows from Besse's book \cite[pg. 271]{Besse} and \cite[Proposition 2.4]{HPW2} that a noncompact $2$-dimensional quasi-Einstein manifold with constant scalar curvature is isometric to one of the following:  $[0,\infty)\times\mathbb{R}$ equipped with the product metric $g=dr^2+dt^2;$ $[0,\infty)\times\mathbb{R}$ with the metric $g=dr^2+\cosh^2(r)dt^2;$ $\mathbb{R}\times\mathbb{R}$ endowed with the metric $g=dr^2+e^{2r}dt^2;$ or the hyperbolic plane $\mathbb{H}^2$ with its standard metric $g=dr^2+\sinh^2(r)dt^2.$ In dimension $n = 3,$ based on the study of B\"ohm metrics \cite{Bohm1,Bohm2}, it is known that there exist more examples of $3$-dimensional quasi-Einstein manifolds than gradient Ricci solitons. According to \cite[Theorem 1.3]{HPW2}, any noncompact $3$-dimensional $m$-quasi-Einstein manifold without boundary, with constant scalar curvature and $m>1$ must be rigid. This, in turn, implies that it is a quotient of either $\mathbb{R}^3,$ or $ \mathbb{R}\times \mathbb{H}^2,$ or $\mathbb{H}^3,$ each endowed with its standard metric. In dimension $n = 4$, however, as shown by He-Petersen-Wylie \cite[Theorem 1.5]{HPW2}, for each $m$, there exist examples of $4$-dimensional quasi-Einstein manifolds with constant scalar curvature that are not rigid. Indeed, the $4$-dimensional case appears to be significantly more intriguing.

Before presenting our first main result, we introduce a new example of a quasi-Einstein manifold with $\lambda < 0$, $\mu < 0$ and constant scalar curvature in arbitrary dimension $n \geq 4$.

\begin{example}
\label{exA}
Let $M^n=\mathbb{H}^{p+1}\times\mathbb{H}^q$, $q>1$, with the product metric 
	\begin{eqnarray*}
		g=dr^2+\sinh^2(r)g_{\mathbb{S}^{p}}+\frac{q-1}{m+p}g_{\mathbb{H}^q},
	\end{eqnarray*} 
	where $r(x,y)=r(x)$ is the height function of $\mathbb{H}^{p+1}$. Considering the potential function $u(r)=\cosh(r)$ and $\lambda=-(m+p)$, one deduces that $(M^n,\,g,\,u,\,\lambda)$ satisfies the system \eqref{eq-qE} and $\mu=-(m-1).$
		\end{example}
			
		We point out that Example \ref{exA} is rigid but not Einstein. A detailed description of Example \ref{exA} will be presented in Section \ref{Sec2}. Taking this new example into account, together with the result by He-Petersen-Wylie \cite[Theorem 1.5]{HPW2} (see also Proposition \ref{prop-hpw}), it appears that a complete classification of noncompact quasi-Einstein manifolds with constant scalar curvature is difficult to achieve.

As was observed in \cite{compact}, on an $m$-quasi-Einstein manifold, we may express the Cotton tensor $C_{ijk}$ and the Weyl tensor $W_{ijkl}$ in terms of an auxiliary $3$-tensor $T_{ijk}$ as follows
	\begin{equation}\label{eq31}
		uC_{ijk}=mW_{ijkl}\nabla_l u+T_{ijk},
	\end{equation} where the $3$-tensor $T_{ijk}$ is given by
\begin{eqnarray}
\label{tensorTT}
		T_{ijk}&=&\frac{m+n-2}{n-2}(R_{ik}\nabla_j u-R_{jk}\nabla_i u)+\frac{m}{n-2}(R_{jl}\nabla_l ug_{ik}-R_{il}\nabla_l ug_{jk})\nonumber\\
		& &+\frac{(n-1)(n-2)\lambda+m R}{(n-1)(n-2)}(\nabla_i ug_{jk}-\nabla_j ug_{ik})\\&&-\frac{u}{2(n-1)}(\nabla_i R g_{jk}-\nabla_j R g_{ik}).\nonumber
	\end{eqnarray} Interestingly, the tensor\footnote{We  emphasize that the tensor $T_{ijk}$ differs from the tensor $D_{ijk}$ introduced in \cite{CH}. Specifically, they are related by $T_{ijk}=\frac{m+n-2}{m}u D_{ijk};$ see Eq. (\ref{eqDTa}).} $T_{ijk}$ shares the same symmetry properties as the Cotton tensor. Furthermore, it is easy to verify that local conformal flatness (i.e., $W_{ijkl} = 0,$ for $n\geq 4$) implies $T_{ijk} = 0.$ Also, by combining \cite[Theorem 1.2]{CH} and (\ref{eq31}), one deduces that any Bach-flat compact $m$-quasi-Einstein manifold with $m>1$ has $T_{ijk} = 0.$ We emphasize that $T_{ijk}$ plays a key role in the study of quasi-Einstein manifolds and is primarily inspired by the tensor $D_{ijk}$, introduced by Cao and Chen in \cite{CC} in the context of gradient Ricci solitons; see also \cite{CaoCat,CH} and \cite{BDE,QY}. The tensor $D_{ijk}$ was fundamental to the classification of Bach-flat gradient Ricci solitons by Cao and Chen \cite{CC} and Cao, Catino, Chen, Mantegazza and Mazzieri \cite{CaoCat}. More recently, Cao and Yu \cite{CYu} classified $D$-flat gradient steady Ricci solitons. In particular, their proof also extends to shrinking and expanding gradient Ricci solitons.

	In the sequel, motived by \cite{CYu}, we establish a classification result for $T$-flat $m$-quasi-Einstein manifolds with constant scalar curvature and $m>1.$ More precisely, we have the following theorem. 

\begin{theorem}
    \label{ThmB}
Let $(M^n,\,g,\,u,\,\lambda),$ $n\geq 3,$ be a nontrivial simply connected $m$-quasi-Einstein manifold with constant scalar curvature $R$ and $m>1.$ Then the tensor $T$ vanishes identically if and only if $(M^n,\,g)$ is isometric to one of the following: 

\begin{enumerate}
\item[(a)] When $\lambda>0$: 
\begin{enumerate}
\item[(i)] the standard hemisphere $\mathbb{S}_+^n;$ 
\item[(ii)] $I\times N$ with the product metric; 
\end{enumerate}
\item[(b)] When $\lambda=0$:
\begin{enumerate}
\item[(iii)]  $[0,\infty)\times F,$ $g=dt^2+g_{F}$ and $u(t,x)=Ct;$
\end{enumerate}
\item[(c)] When $\lambda<0$:
\begin{enumerate}
\item[(iv)] the hyperbolic space $\mathbb{H}^n,$ $g=dt^2+\sqrt{-k}\sinh^2(\sqrt{-k} t)g_{\mathbb{S}^{n-1}}$ and $u=C\cosh (\sqrt{-k}t);$
\item[(v)]  $[0,\infty)\times N,$ $g=dt^2+\sqrt{-k}\cosh^2(\sqrt{-k}t)g_{\mathbb{S}^{n-1}}$ and $u(t,x)=C\sinh(\sqrt{-k}t);$ 
\item[(vi)]  $\mathbb{R}\times F,$ $g=dt^2+e^{2\sqrt{-k}t}g_{F}$ and $u(t,x)=Ce^{2\sqrt{-k}t};$ 
\item[(vii)] $\mathbb{R}\times N$ equipped with product metric and $u(t,x)=\cosh(t);$ 
\item[(viii)] $\mathbb{R}\times N$ equipped with product metric and $u(t,x)=e^{t};$ 
\item[(ix)] $[0,\infty)\times N$ with the product metric and $u(t,x)=\sinh(t);$  
\end{enumerate}
\end{enumerate}
where $F$ is Ricci flat, $C$ is
an arbitrary positive constant and $N$ is a $\lambda$-Einstein manifold.
\end{theorem}

We note that Example \ref{exA} does not satisfy the $T$-flat condition. As an application of Theorem \ref{ThmB}, we obtain the following classification result for noncompact $3$-dimensional $m$-quasi-Einstein manifolds, which includes the case of non-empty boundary.

\begin{theorem}
    \label{ThmA}
    Let $(M^3,\,g,\,u,\,\lambda)$ be a noncompact nontrivial $3$-dimensional simply connected (possibly with boundary) $m$-quasi-Einstein manifolds with $m>1.$ Then $M^3$ has constant scalar curvature if and only if it is isometric to either 
    \begin{itemize}
        \item[(i)] the standard hyperbolic space $\mathbb{H}^3,$ or
        \item[(ii)] $[0,\infty)\times \mathbb{H}^2$ with the warped product metric $g=dt^2+\cosh^{2}(t)g_{\mathbb{H}^2},$ or
        \item[(iii)] $[0,\infty)\times \mathbb{H}^2$ with the product metric, or
        \item[(iv)] $\mathbb{R}\times \mathbb{R}^2$ with the warped product metric $g=dt^2+e^{2\sqrt{-k}t}g_{\mathbb{R}^2}$, or
        \item[(v)] $\mathbb{R}\times \mathbb{H}^2$ with the product metric $g=dt^2+\frac{1}{(-\lambda)}g_{\mathbb{H}^{2}}$, or
        \item[(vi)] $[0,\infty)\times \mathbb{R}^2$ with the product metric $g=dt^2+g_{\mathbb{R}^2}.$
    \end{itemize} 
\end{theorem}

Next, as a direct consequence of Theorem \ref{ThmA} and \cite[Theorem 3]{costa-ribeiro-zhou}, we get the following corollary.

\begin{corollary}
Let $(M^3,\,g,\,u,\,\lambda)$ be a nontrivial $3$-dimensional simply connected (possibly with boundary) $m$-quasi-Einstein manifolds with $m>1.$ Then $M^3$ has constant scalar curvature if and only if it is isometric to either 

    \begin{itemize}
         \item[(i)] the standard hemisphere $\mathbb{S}^{3}_{+}$, or
        \item[(ii)] the cylinder $I\times\mathbb{S}^2$ with the product metric, or
        \item[(iii)] the standard hyperbolic space $\mathbb{H}^3,$ or
        \item[(iv)] $[0,\infty)\times \mathbb{H}^2$ with the warped product metric $g=dt^2+\cosh^{2}(t)g_{\mathbb{H}^2},$ or
        \item[(v)] $[0,\infty)\times \mathbb{H}^2$ with the product metric, or
        \item[(vi)] $\mathbb{R}\times \mathbb{R}^2$ with the warped product metric $g=dt^2+e^{2\sqrt{-k}t}g_{\mathbb{R}^2}$, or
        \item[(vii)] $\mathbb{R}\times \mathbb{H}^2$ with the product metric $g=dt^2+\frac{1}{(-\lambda)}g_{\mathbb{H}^{2}}$, or
        \item[(viii)] $[0,\infty)\times \mathbb{R}^2$ with the product metric $g=dt^2+g_{\mathbb{R}^2}.$
    \end{itemize} 
\end{corollary}

\medskip

The remainder of this paper is organized as follows. In Section \ref{Sec2}, we review some basic facts about $m$-quasi-Einstein manifolds, present several foundational lemmas and discuss Example \ref{exA}. In Section \ref{sec3}, we establish key results and derive novel formulas that will be used in the proofs of Theorems \ref{ThmB} and \ref{ThmA}. Section \ref{Sec4} contains the proofs of Theorems \ref{ThmB} and \ref{ThmA}.

\medskip

{\bf Acknowledgement.} The authors would like to thank Detang Zhou and Rafael Di\'ogenes for enlightening conversations and helpful suggestions.

\section{Background and Examples}
\label{Sec2}

In this section, we review some basic facts and present several lemmas that will be useful in the proofs of the main results. Additionally, we introduce new examples of noncompact quasi-Einstein manifolds.

We start by recalling that an $m$-quasi-Einstein manifold is a Riemannian manifold $(M^n,\,g)$, $n\geq 2$, endowed with a smooth function $u:M\to \Bbb{R}$ satisfying

    \begin{eqnarray}
    \label{eq-fundamental}
        \nabla^2 u=\frac{u}{m}\left(Ric-\lambda g\right).
    \end{eqnarray} In particular, taking the trace of (\ref{eq-fundamental}), we obtain 
    
    \begin{equation}
    \label{eqtracefund}
    \Delta u=\frac{u}{m}\left(R -\lambda n\right).
    \end{equation} 
Regarding the regularity of the potential function, it is known that, for an $m$-quasi-Einstein manifold $(M^n, g, u, \lambda)$, both $u$ and $g$ are real analytic in harmonic coordinates (cf. Proposition 2.4 in \cite{HPW}).

We also recall some important properties of $m$-quasi-Einstein manifolds (cf. \cite{CaseShuWei,compact,kk}).

\begin{proposition}
\label{propA}
	Let $(M^n,\,g,\,u,\,\lambda),$ $n\geq 3,$ be an $n$-dimensional $m$-quasi-Einstein manifold. Then we have:
	\begin{itemize}
	\item[]
\begin{equation}
\label{eqK1a}
	\frac{1}{2}u\nabla R=-(m-1)Ric(\nabla u)-(R-(n-1)\lambda)\nabla u;
\end{equation}

	\item[]
	\begin{eqnarray}
	\label{laplacianoR}
		\frac{1}{2}\Delta R+\frac{m+2}{2u}\langle \nabla u,\nabla R\rangle&=& -\frac{m-1}{m}\left|Ric-\frac{R}{n}g\right|^2\nonumber\\
		&&-\frac{(n+m-1)}{nm}(R-n\lambda)\left(R-\frac{n(n-1)}{n+m-1}\lambda\right).
	\end{eqnarray}
	\end{itemize}
\end{proposition}

It follows from (\ref{eqK1a}) that an $m$-quasi-Einstein manifold $M^n$ with constant scalar curvature and $m>1$ must satisfy 
\begin{eqnarray}\label{eq-ridu}
    Ric(\nabla u)=\frac{(n-1)\lambda- R}{m-1}\nabla u.
\end{eqnarray} Consequently, 
\begin{eqnarray}\label{ritdu}
    \mathring{Ric}(\nabla u)=\frac{n(n-1)\lambda-(m+n-1)R}{n(m-1)}\nabla u,
\end{eqnarray} where $\mathring{Ric}$ denotes the traceless Ricci tensor. Moreover, by using (\ref{laplacianoR}), one deduces that an $m$-quasi-Einstein manifold $M^n$ with constant scalar curvature and $m>1$ also satisfies 

\begin{eqnarray}
\label{normricdu}
|\mathring{Ric}|^2 = \frac{1}{n(m-1)}\Big(R-n\lambda\Big)\Big(n(n-1)\lambda -(m+n-1)R\Big).
\end{eqnarray}

From (\ref{normricdu}), we obtain the following result (cf. \cite{CaseShuWei} and \cite{HPW2}). 

\begin{proposition}
    \label{propL0}
    Let $(M^n,\,g,\,u),$ $n\geq 3,$ be a nontrivial $n$-dimensional $m$-quasi-Einstein manifold with $\lambda=0$ and $m>1.$ If $M^n$ has constant scalar curvature, then $M^n$ is isometric to $[0,\infty)\times F$ with the metric $g=dt^2 +g_{F},$ where $F$ is Ricci flat.  
\end{proposition}

\begin{proof}
  Since $M^n$ has constant scalar curvature and $\lambda=0,$ Eq. (\ref{normricdu}) yields

    \begin{equation*}
        (m-1)\left|Ric-\frac{R}{n}g\right|^2 =-\frac{(m+n-1)}{n}R^2,
    \end{equation*} which implies that $M^n$ is Ricci flat. The conclusion then follows from \cite[Proposition 2.4]{HPW2}, which asserts that $M^n$ is isometric to $[0,\infty)\times F$ with the metric $g=dt^2 +g_{F},$ where $F$ is Ricci flat. This completes the proof.
\end{proof}

We now recall some important tensors that arise in the study of curvature on a Riemannian manifold $(M^n,\,g)$ of dimension $n\ge 3.$  The first one is the {\it Weyl tensor} $W$ which is defined by the following decomposition formula:
\begin{eqnarray}
\label{weyl}
R_{ijkl}&=&W_{ijkl}+\frac{1}{n-2}\big(R_{ik}g_{jl}+R_{jl}g_{ik}-R_{il}g_{jk}-R_{jk}g_{il}\big) \nonumber\\
 &&-\frac{R}{(n-1)(n-2)}\big(g_{jl}g_{ik}-g_{il}g_{jk}\big),
\end{eqnarray} where $R_{ijkl}$ stands for the Riemann curvature tensor. It is well known that $W=0$ in dimension $n=3.$ The second one is the {\it Cotton tensor} $C$ given by
\begin{equation}
\label{cotton} \displaystyle{C_{ijk}=\nabla_{i}R_{jk}-\nabla_{j}R_{ik}-\frac{1}{2(n-1)}\big(\nabla_{i}R
g_{jk}-\nabla_{j}R g_{ik}).}
\end{equation} Notice that $C_{ijk}$ is skew-symmetric in the first two indices and trace-free in any two indices. Moreover, for $n\geq 4,$ we have 

\begin{equation}
C_{ijk}=-\frac{(n-2)}{(n-3)}\nabla_{l}W_{ijkl}.
\end{equation} While the {\it Schouten tensor} $A$ is defined by
\begin{equation}
\label{schouten} A_{ij}=R_{ij}-\frac{R}{2(n-1)}g_{ij}.
\end{equation} Combining Eqs. (\ref{weyl}) and (\ref{schouten}), one obtains that 
\begin{equation}
\label{WS} R_{ijkl}=\frac{1}{n-2}(A\odot g)_{ijkl}+W_{ijkl},
\end{equation} where $\odot$ is the Kulkarni-Nomizu product, which takes two symmetric $(0,2)$-tensors and yields a $(0,4)$-tensor with the same algebraic symmetries as the curvature tensor, given by
\begin{eqnarray}
\label{eq76}
(\alpha \odot \beta)_{ijkl}=\alpha_{ik}\beta_{jl}+\alpha_{jl}\beta_{ik}-\alpha_{il}\beta_{jk}-\alpha_{jk}\beta_{il}.
\end{eqnarray} 

As previously mentioned, the Cotton tensor can be expressed in terms of the Weyl tensor and the $3$-tensor $T_{ijk}$ as follows (see \cite[Lemma 2]{compact}).

\begin{lemma}[\cite{compact}]
	\label{lem1}
	Let $(M^n,\,g,\,u,\,\lambda)$ be an $m$-quasi-Einstein manifold. Then it holds
	\begin{equation}\label{eq3}
		uC_{ijk}=mW_{ijkl}\nabla_l u+T_{ijk},
	\end{equation} where the $3$-tensor $T_{ijk}$ is given by
	\begin{eqnarray}
	\label{as}
		T_{ijk}&=&\frac{m+n-2}{n-2}(R_{ik}\nabla_j u-R_{jk}\nabla_i u)+\frac{m}{n-2}(R_{jl}\nabla_l ug_{ik}-R_{il}\nabla_l ug_{jk})\nonumber\\
		& &+\frac{(n-1)(n-2)\lambda+m R}{(n-1)(n-2)}(\nabla_i ug_{jk}-\nabla_j ug_{ik})\\&&-\frac{u}{2(n-1)}(\nabla_i R g_{jk}-\nabla_j R g_{ik}).\nonumber
	\end{eqnarray}
\end{lemma}

Substituting (\ref{eqK1a}) into (\ref{as}), it is not difficult to check that

\begin{equation}
\label{eqDTa}
T_{ijk}=\left(\frac{m+n-2}{m}\right)u D_{ijk},
\end{equation} where $D_{ijk}$ is the $3$-tensor defined in \cite{CH} as 

\begin{eqnarray}
D_{ijk}&=&\frac{1}{n-2}\left(R_{jk}\nabla_{i}f-R_{ik}\nabla_{j}f\right)\nonumber\\&& + \frac{1}{(n-1)(n-2)}\left(R_{il}\nabla_{l}f g_{jk}-R_{jl}\nabla_{l}f g_{ik}\right)\nonumber\\&&-\frac{R}{(n-1)(n-2)}\left(g_{jk}\nabla_{i}f - g_{ik}\nabla_{j}f\right).
\end{eqnarray}

In order to proceed, for $m>1$, we recall the auxiliary tensor $P$ defined by
\begin{equation}
\label{ppa}
    P=Ric-\left(\frac{(n-1)\lambda-R}{m-1}\right) g.
\end{equation} In particular, following \cite{HPW}, we consider the $4$-tensor $Q$ associated with $P$ as  
\begin{eqnarray}
\label{tensorQ}
    Q=Rm+\frac{1}{m}P\odot g+\frac{(n-m)\lambda-R}{2m(m-1)}g\odot g,
\end{eqnarray} where $\odot$ stands for the Kulkarni-Nomizu product, as defined in $\eqref{eq76}$, and $Rm$ is the Riemann curvature tensor.

The next lemma provides several useful formulas for quasi-Einstein manifolds in terms of the auxiliary tensor $P$ (see also \cite[Lemma 3.2]{CaseShuWei} and \cite[Proposition 3.4]{HPW2}). 

\begin{lemma}\label{lemma-aux}
    Let $(M^n,g,u,\lambda)$ be an $m$-quasi-Einstein manifold with $m>1.$ Then we have:
    \begin{align}\label{l1}
        P(\nabla u)&=-\frac{1}{2(m-1)}u\nabla R;\\\label{l2}
        \frac{u}{2} L_{m+2}(R)&=\frac{m-1}{m}Tr\left(P\circ((\lambda-\rho)I-P)\right);\\\label{l23}
        u(\nabla_i P_{jk}-\nabla_j P_{ik})&=m Q_{ijkl}\nabla_l u+\frac{1}{2}(g\odot g)_{ijkl}P_{sl}\nabla_l u,
    \end{align}
    where $\rho=\frac{(n-1)\lambda-R}{m-1}$ and $uL_{m+2}(f)=u\Delta f+(m+2)\langle \nabla u,\nabla f \rangle$. In particular, if the scalar curvature is constant, the following additional identities hold:

    \begin{align}\label{l4}
        Q_{ijkl}\nabla_l u&=\frac{u}{m}(\nabla_i R_{jk}-\nabla_j R_{ik});\\\label{l5}
        \frac{u}{m}\nabla_l P_{ij}\nabla_l u&=\left(\frac{u}{m}\right)^2\left(-\left(\lambda-\rho\right)P_{ij}+P_{ik}P_{kj}\right)+Q_{kijl}\nabla_k u\nabla_l u.
    \end{align}
\end{lemma}

\begin{proof}
To prove (\ref{l1}), it suffices to combined \eqref{ppa} with \eqref{eqK1a}.

Next, since $uL_{m+2}(f)=u\Delta f+(m+2)\langle \nabla u,\nabla f \rangle$, for any smooth real-valued functions $f$ on $M^n$, we deduce that
\begin{eqnarray*}
    \frac{u}{2} L_{m+2}(R)=\frac{u}{2}\Delta R+\frac{m+2}{2}\langle \nabla u,\nabla R \rangle.
\end{eqnarray*}
By \cite[Lemma 3.2]{CaseShuWei}, this expression can be rewritten as

\begin{eqnarray*}
    \frac{u}{2} L_{m+2}(R)=\frac{m-1}{m}Tr(Ric\circ(\lambda I-Ric))-\frac{1}{m}(R-n\lambda)(R-(n-1)\lambda),
\end{eqnarray*} where $\circ$ denotes the matrix product, i.e.,  for any $2$-tensors $S$ and $T$, we set $S\circ T=(S_{ij}T_{jk})_{ik}.$ Using the definition of the tensor $P$ and the fact that $R = -(m - 1) \rho + (n - 1) \lambda,$ it follows that

\begin{eqnarray*}
    \frac{u}{2} L_{m+2}(R)&=&\frac{m-1}{m}Tr\left((P+\rho I)\circ(\lambda I-(P+\rho I))\right)\nonumber\\&&-\frac{m-1}{m}((m-1)\rho+\lambda))\rho\\
    &=&\frac{m-1}{m}Tr\left(P\circ((\lambda-\rho)I-P)\right)\\
    & &+\frac{m-1}{m}\left(n\lambda\rho-\rho Tr(P)-n\rho^2-((m-1)\rho+\lambda))\rho\right)\\
    &=&\frac{m-1}{m}Tr\left(P\circ((\lambda-\rho)I-P)\right)\\
    & &+\frac{m-1}{m}\Big(\rho[(n-1)\lambda-(m+n-1)\rho-Tr(P)]\Big)\\
    &=&\frac{m-1}{m}Tr\left(P\circ((\lambda-\rho)I-P)\right),
\end{eqnarray*} where we have used the relation $Tr(P)=-(m+n-1)\rho+(n-1)\lambda,$ which concludes the proof of $\eqref{l2}.$

The proofs of the remaining identities were presented in Propositions 3 and 4 of \cite{costa-ribeiro-zhou}. For the sake of completeness, we include them here. To prove \eqref{l23}, we begin by rewriting the identity from \cite[Lemma 1]{compact}:
\begin{eqnarray*}
u\left(\nabla_{i}R_{jk}-\nabla_{j}R_{ik}\right)&=& m R_{ijkl}\nabla_{l}u+\lambda\left(\nabla_{i}ug_{jk}-\nabla_{j}u g_{ik}\right)\nonumber\\&&-\left(\nabla_{i}u R_{jk}-\nabla_{j}u R_{ik}\right).
\end{eqnarray*} Expressing this in terms of the tensor $P = \mathrm{Ric} - \rho g$, we obtain
\begin{eqnarray*}
u\left(\nabla_{i}P_{jk}-\nabla_{j}P_{ik}\right)&+&u\left(\nabla_{i}\rho g_{jk} -\nabla_{j}\rho g_{ik}\right)\nonumber\\&=& mR_{ijkl}\nabla_{l}u+(\lambda -\rho)\left(\nabla_{i}u g_{jk}-\nabla_{j}u g_{ik}\right)\nonumber\\&&-\left(\nabla_{i}u P_{jk}-\nabla_{j}u P_{ik}\right). 
\end{eqnarray*} Moreover, by Eq. $\eqref{l1}$ and noting that $\rho = \frac{(n - 1) \lambda - R}{m - 1}$, one sees that $\frac{u}{2}\nabla \rho=P(\nabla u)$ (see also \cite[Proposition 5.2]{HPW}). Hence,
\begin{eqnarray}\nonumber
u\left(\nabla_{i}P_{jk}-\nabla_{j}P_{ik}\right)&+&2\left(P_{il}\nabla_{l}u g_{jk} -P_{jl}\nabla_{l}u g_{ik}\right)\nonumber\\&=& mR_{ijkl}\nabla_{l}u+(\lambda -\rho)\left(\nabla_{i}u g_{jk}-\nabla_{j}u g_{ik}\right)\nonumber\\&&-\left(\nabla_{i}u P_{jk}-\nabla_{j}u P_{ik}\right).\label{lmn450} 
\end{eqnarray}  

On the other hand, it follows from the definition of the Kulkarni-Nomizu product that

\begin{equation*}
(g\odot g)_{ijkl}\nabla_{l}u=2(g_{ik}\nabla_{j}u-g_{jk}\nabla_{i}u),
\end{equation*}
\begin{equation*}
(g\odot g)_{ijkl}P_{sl}\nabla_{s}u=2(P_{js}\nabla_{s}u g_{ik}-P_{is}\nabla_{s}u g_{jk})
\end{equation*} 
and
\begin{equation*}
(P\odot g)_{ijkl}\nabla_{l}u=(P_{ik}\nabla_{j}u-P_{jk}\nabla_{i}u)+(P_{jl}\nabla_{l}u g_{ik}-P_{il}\nabla_{l}u g_{jk}).
\end{equation*} 
Substituting these expressions into (\ref{lmn450}), we deduce
\begin{eqnarray*}
u(\nabla_{i}P_{jk}-\nabla_{j}P_{ik})&=&mR_{ijkl}\nabla_{l}u+(\rho-\lambda)(g_{ik}\nabla_{j}u-g_{jk}\nabla_{i}u)\nonumber\\&&+(P_{ik}\nabla_{j}u-P_{jk}\nabla_{i}u)+ 2(P_{jl}\nabla_{l}u g_{ik} - P_{il}\nabla_{l}u g_{jk})\nonumber\\&=& mR_{ijkl}\nabla_{l}u +\frac{(\rho-\lambda)}{2}(g\odot g)_{ijkl}\nabla_{l}u + (P\odot g)_{ijkl}\nabla_{l}u\nonumber\\&&+\frac{1}{2}(g\odot g)_{ijkl}P_{sl}\nabla_{s}u\nonumber\\&=& mQ_{ijkl}\nabla_{l}u+\frac{1}{2}(g\odot g)_{ijkl}P_{sl}\nabla_{s}u,
\end{eqnarray*} where the final equality follows from \eqref{tensorQ}, thus proving \eqref{l23}.

Finally, equations \eqref{l4} and \eqref{l5} follow directly from \eqref{l23}, since the scalar curvature is constant, $P(\nabla u) = 0$ and consequently

\begin{eqnarray*}
0=\nabla_{j}\left(P_{ik}\nabla_{i}u\right)=(\nabla_{j}P_{ik})\nabla_{i}u + P_{ik}\nabla_{j}\nabla_{i}u.
\end{eqnarray*}
This completes the proof of the lemma.

\end{proof}

Next, it is also useful to recall the expression for the Ricci tensor of a warped product in terms of the base and fiber metrics. The following proposition corresponds to \cite[Corollary 43]{O'Neil}

\begin{proposition}[\cite{O'Neil}]\label{prop-oneil}
	The Ricci curvature of a warped product manifold $M=B\times_\varphi F$ with $m=dim(F)$, $X, Y$ and $Z, V$ any horizontal and vertical vectors, respectively, satisfies:
	\begin{itemize}
		\item[(i)] $Ric(X,Y)=Ric_B(X,Y)-\frac{m}{\varphi}\nabla^{2}_{g_B}\varphi(X,Y)$,
		\item[(ii)] $Ric(X,V)=0$,
		\item[(iii)] $Ric(Z,V)=Ric_F(Z,V)-(\varphi\Delta_{g_B} \varphi+(m-1)|\nabla \varphi|^{2}_{g_B})g_F(Z,V)$.
	\end{itemize}
\end{proposition}

We now proceed to present all details of Example \ref{exA}.

\begin{example}[Example \ref{exA}]\label{ex2}
	Given $M^n=\mathbb{H}^{p+1}\times\mathbb{H}^q$, $q>1$, endowed with the metric 
	\begin{eqnarray*}
		g=dr^2+\sinh^2(r)g_{\mathbb{S}^{p}}+\frac{q-1}{m+p}g_{\mathbb{H}^q},
	\end{eqnarray*} 
	where $r(x,y)=r(x)$ is the height function of $\mathbb{H}^{p+1}.$ Consider $u(r)=\cosh(r)$ and $\lambda=-(m+p)$.

To verify that $M^n$ is indeed a quasi-Einstein manifold with $\lambda < 0$, we first observe that
\begin{eqnarray*}
	\nabla u=\sinh(r)\nabla r.
\end{eqnarray*} From this, using the Lie derivative, we obtain
\begin{eqnarray*}
	\nabla^2 u=\cosh(r)\left(dr^2+\sinh^2(r)g_{\mathbb{S}^p}\right).
\end{eqnarray*} Therefore, $$\nabla^2 u=u g_{\mathbb{H}^{p+1}}.$$ 
Next, since $g=g_{\mathbb{H}^{p+1}}+\frac{q-1}{p+1}g_{\mathbb{H}^q}$, we can express the Ricci tensor as
	\begin{eqnarray*}
		Ric=-p g_{\mathbb{H}^{p+1}}-(q-1) g_{\mathbb{H}^q}.
	\end{eqnarray*}  
Given that $\lambda = -(m + p)$, it follows that
	\begin{eqnarray*}
		Ric-\lambda g&=&-p g_{\mathbb{H}^{p+1}}-(q-1) g_{\mathbb{H}^{q}}+(m+p)\left(g_{\mathbb{H}^{p+1}}+\frac{q-1}{m+p}g_{\mathbb{H}^q}\right)\\
        &=& m g_{\mathbb{H}^{p+1}}.
	\end{eqnarray*} 
Thus, we arrive at
	\begin{equation*}
		\nabla^2 u=\frac{u}{m}(Ric-\lambda g),
	\end{equation*} which confirms that $\mathbb{H}^{p+1} \times \mathbb{H}^q$ is a quasi-Einstein manifold with $\lambda < 0$.

Furthermore, the scalar curvature is given by
\begin{eqnarray}\label{ex-R1}
R = \frac{q(m - n) + n(n - 1)}{m + n - q - 1} \lambda.
\end{eqnarray}
\end{example}

We point out that Example \ref{exA} is rigid but not Einstein. Moreover, we recall the following interesting result by He-Petersen-Wylie \cite{HPW2}. 

\begin{proposition}[cf. \cite{HPW2}]
	\label{prop-hpw}
	A non-trivial complete rigid $m$-quasi-Einstein manifold $(M^n,\,g,\,u,\,\lambda)$ is one of the examples in \cite[Proposition 2.4]{HPW2}, or its universal cover splits off as $$\widetilde{M}=(M_1,g_1)\times(M_2,g_2)\,\,\,\,\,\,\hbox{with}\,\,\,\,\,\,\,\,u(x,y)=u(y),$$ where $(M_{1},g_1,\,\lambda)$ is a trivial quasi-Einstein manifold and $(M_{2},\,g_2,\,u)$ is one of the examples in \cite[Proposition 2.4]{HPW2}.  
\end{proposition}

To conclude this section, we present additional examples of rigid quasi-Einstein manifolds with $\lambda < 0$ that are also non-Einstein. The verification of these examples follows a similar approach to that used in Example \ref{exA}.   

\begin{proposition}[Nontrivial rigid quasi-Einstein manifolds with $\lambda<0,$ which are not Einstein]\label{prop-rigid} Let $(M^n,g,u,\lambda)$ be a simply connected nontrivial $m$-quasi-Einstein manifold with $\lambda<0$ which is not Einstein. Then, it is rigid if and only if it is one of the following structures: 
    \begin{itemize}
        \item[(a)] $[0,\infty)\times N$ with the metric $g=dr^2+g_N$, $\lambda=-m$ and $u(r,x)=\sinh(r),$
        \item[(b)] $\mathbb{R}\times N$ with the metric $g=dr^2+g_N$, $\lambda=-m$ and $u(r,x)=e^r,$
        \item[(c)] $\mathbb{R}\times N$ with the metric  $g=dr^2+g_N$, $\lambda=-m$ and $u(r,x)=\cosh(r),$
        \item[(d)] $[0,\infty)\times A\times N$ with the metric $g=dr^2+\cosh^2(r)g_A+g_N$, $Ric_A=-(n-q-2)$, $\lambda=-(m+n-q-1)$ and $u(r,y,x)=\sinh(r),$
        \item[(e)] $\mathbb{R}\times F\times N$ with the metric $g=dr^2+e^{2r}g_F+g_N$, $Ric_F=0$, $\lambda=-(m+n-q-1)$ and $u(r,y,x)=e^r,$
        \item[(f)] $\mathbb{H}^n\times N$ with the metric $g=dr^2+\sinh^2(r)g_{\mathbb{S}^{n-1}}+g_N$, $\lambda=m+n-q-1$ and $u(r)=\cosh(r),$
    \end{itemize} where $N^q$, $q\geq 2$, is $\lambda$-Einstein.
\end{proposition}

\section{Some Key Results}
\label{sec3}
Throughout this section, we establish several key results that will be used in the proofs of Theorems \ref{ThmB} and \ref{ThmA}. In particular, we derive novel formulas and identify a set of possible values for the constant scalar curvature of noncompact quasi-Einstein manifolds with $\lambda<0.$

We begin by recalling that a non-constant function $f : M \to \mathbb{R}$ of class at least $C^2$ is said to be {\it transnormal} if
\begin{eqnarray}\label{eqt}
|\nabla f|^2 = b(f),
\end{eqnarray}
for some $C^2$ function $b$ defined on the range of $f$ in $\mathbb{R}$. Additionally, $f$ is said to be {\it isoparametric} if there exists a continuous function $a$, also defined on the range of $f$, such that
\begin{eqnarray}\label{eqi}
\Delta f = a(f).
\end{eqnarray} We refer the reader to \cite{GT,GT2,Wang} for further details.

From now on, consider an $m$-quasi-Einstein manifold $(M^n,\, g,\, u,\, \lambda)$ with constant scalar curvature $R.$ Since $m>1,$ one obtains from Eq. (\ref{eqmu}) that
\begin{eqnarray}
\label{transnormal}
|\nabla u|^2 = \frac{\mu}{m - 1} - \frac{R + (m - n)\lambda}{m(m - 1)} u^2.
\end{eqnarray} Consequently, the potential function $u$ is transnormal, with
\begin{eqnarray}
\label{eqb}
b(u) = \frac{\mu}{m - 1} - \frac{R + (m - n)\lambda}{m(m - 1)} u^2.
\end{eqnarray} In light of this, it is straightforward to verify from (\ref{eqtracefund}) that the potential function $u$ is also isoparametric.

We now recall the following useful lemma established in \cite[Lemma 23]{costa-ribeiro-zhou}.

\begin{lemma}[\cite{costa-ribeiro-zhou}]
\label{TRic}
Let $(M^n,g,u,\lambda)$ be an $m$-quasi-Einstein manifold with constant scalar curvature. Then we have:
\begin{eqnarray}
\label{righ}
    \mathring{R}_{ik}T_{ijk}\nabla_j u&=&\frac{(m+n-2)}{(n-2)}|\mathring{Ric}|^2|\nabla u|^2-\frac{(m+n-2)}{(n-2)}\nabla_j u\mathring{R}_{ik}\mathring{R}_{jk}\nabla_i u\nonumber\\
    & &+\frac{n(n-1)\lambda-(m+n-1)R}{n(n-1)}\mathring{Ric}(\nabla u,\nabla u)\nonumber\\&&-\frac{m}{n-2}\nabla_j u\mathring{R}_{ij}\mathring{R}_{il}\nabla_l u\nonumber\\
    &=& \frac{(n-2)}{2(m+n-2)}|T|^2.
\end{eqnarray} 
\end{lemma}

As already mentioned, the $3$-tensor $T_{ijk}$ plays an important role in the study of quasi-Einstein manifolds. In particular, assuming that a quasi-Einstein manifold is $T$-flat, the authors obtained the following result.

\begin{proposition}[\cite{costa-ribeiro-zhou}]
\label{prop1-p}
	Let $(M^n,\,g,\,u,\,\lambda)$ be an $m$-quasi-Einstein manifold with constant scalar curvature and $m>1.$ Then $T$ vanishes identically if and only if the Ricci tensor has at most two distinct eigenvalues, one of which has multiplicity at least $n - 1$, and its eigenspace corresponds to the orthogonal complement of $\nabla u$.
\end{proposition}

Next, as a direct consequence of Lemma \ref{TRic}, we obtain the following proposition.

\begin{proposition}
\label{propT0}
    Let $(M^n,g,u,\lambda)$ be a nontrivial $m$-quasi-Einstein manifold with constant scalar curvature $R$ and $m>1.$  Then $T$ is identically zero if and only if either $R=\frac{n(n-1)}{(m+n-1)}\lambda$ or $R=(n-1)\lambda.$
\end{proposition}
\begin{proof}
Substituting (\ref{ritdu}) into Lemma \ref{TRic}, one obtains that an $m$-quasi-Einstein manifold with constant scalar curvature must satisfy

\begin{eqnarray}
\frac{(n-2)}{2(m+n-2)}|T|^2&=&\frac{(m+n-2)}{(n-2)} |\mathring{Ric}|^{2} |\nabla u|^2 \nonumber\\ &&- \theta\Big(n(n-1)\lambda-(m+n-1)R\Big)^{2} |\nabla u|^2,
\end{eqnarray} where $\theta=\frac{(m+n-2)}{n(n-1)(n-2)(m-1)^2}.$ Hence, by using (\ref{normricdu}), one sees that 

\begin{equation}
\frac{(n-2)^2}{2(m+n-2)^2}|T|^2= \overline{\theta}\Big(R-(n-1)\lambda\Big)\Big(n(n-1)\lambda-(m+n-1)R\Big) |\nabla u|^2,
\end{equation} where $\overline{\theta}=\frac{m}{(n-1)(m-1)^2}.$ This proves the stated result. 
\end{proof}

Inspired by \cite{costa-ribeiro-zhou} and following the approach of \cite{Fl-Gr}, we provide a list of all possible values for the constant scalar curvature $R$ of a noncompact quasi-Einstein manifold with $\lambda < 0$ and $\mu < 0$. To do so, we consider the set of minimum points of $u$ given by
\begin{eqnarray*}
    MIN(u)=\{p\in M:\, u(p)=u_{\min}\}.
\end{eqnarray*} Now, we may state our next result.

\begin{theorem}
	\label{theo1-i}
	Let $(M^n,\,g,\,u,\,\lambda)$ be an $m$-quasi-Einstein manifold with $\lambda<0$, $\mu<0$, $m>1$ and constant scalar curvature $R.$ Then we have:
	\begin{eqnarray}
		\label{estR-i}
		R\in\left\{\frac{n(n-1)}{m+n-1}\lambda,\,\frac{m+n(n-2)}{m+n-2}\lambda,\,\ldots,\,(n-1)\lambda,\,n\lambda \right\}.
	\end{eqnarray}
	In general, one has $R=\frac{\kappa(m-n)+n(n-1)}{m+n-\kappa-1}\lambda,$ for some $\kappa\in\{0,1,\ldots,n-1,n\}.$
\end{theorem}

\begin{proof}
To begin with, since $\lambda < 0$, it follows from \cite[Remark 3.12]{HPW2}, in conjunction with \cite[Proposition 3.11]{HPW2}, that there is only one case to consider, namely, when the potential function has a non-empty critical point set, which corresponds to the condition $\mu < 0$. In this setting, we may take $$u(r)=\cosh(\sqrt{-\alpha}r),$$ where $\alpha=\frac{R+(m-n)\lambda}{m(m-1)}.$ As the potential function admits only minimum values, we deduce that the critical set of $u$ coincides with the set of minimum values, that is, $Crit(u)=MIN(u).$ Thus, we may identify $MIN(u)=r^{-1}(0).$ 

Following the argument in \cite[Lemma 7]{Wang},  with suitable adaptations, and using the fact that $u$ vanishes on each boundary component, we conclude that each connected component of $MIN(u)$ is a smooth submanifold. Consequently, from \cite[Lemma 9]{costa-ribeiro-zhou}, it follows that
	\begin{eqnarray}\label{laplacian r}
		\Delta r=Tr(\mathcal{A}_\eta)+\frac{n-\kappa-1}r+O(r),
	\end{eqnarray} where $\kappa$ is the dimension of a connected component $N$ of $MIN(u),$ and $\mathcal{A}_{\eta}$ is the second fundamental form with respect to the normal vector $\eta.$ 
	
	Given $u(r)=\cosh(\sqrt{-\alpha}r)$, we compute
	\begin{eqnarray*}
		\nabla_i\nabla_j u=\sqrt{-\alpha}\sinh(\sqrt{-\alpha} r)\nabla_i\nabla_j r-\alpha\cosh(\sqrt{-\alpha}r)\nabla_i r\nabla_j r 
	\end{eqnarray*}
	and 
	\begin{eqnarray}\label{delta-u}
		\Delta u=\sqrt{-\alpha}\sinh(\sqrt{-\alpha}r)\Delta r-\alpha\cosh(\sqrt{-\alpha}r)|\nabla r|^2.
	\end{eqnarray} So, by using the Taylor expansions, around $r=0,$ 
	\begin{eqnarray*}
		\sinh(\sqrt{-\alpha}r)=\sqrt{-\alpha} r+O(r^3)\;\,\,\,\,\mathrm{and}\;\,\,\,\,\,\cosh(\sqrt{-\alpha}r)=1+O(r^2),
	\end{eqnarray*} and applying \eqref{laplacian r} and \eqref{delta-u}, one sees that
	\begin{eqnarray}\nonumber
		\Delta u&=&(-\alpha r+O(r^3))\left(\textrm{Tr}(\mathcal{A}_\eta)+\frac{n-\kappa-1}r+O(r)\right)+(-\alpha+O(r^2))\\\label{du-o}
		&=&-\alpha(n-\kappa)+O(r),
	\end{eqnarray} where we also used that $r$ is a distance function, so $|\nabla r|=1$
    
	Next, from \eqref{ppa}, we know that $$P=Ric-\left(\frac{(n-1)\lambda-R}{m-1}\right)g$$ and defining $\rho=\frac{(n-1)\lambda-R}{m-1},$ we rewrite (\ref{eqtracefund}) on the connected component $N\subset MIN(u)$ as

	\begin{eqnarray}\label{du-trp}
		\Delta u=\frac{1}{m}(Tr(P)-n(\lambda-\rho)),
	\end{eqnarray} where we have used that $u\mid_{N}=1.$ Since $\alpha=\frac{\lambda-\rho}{m},$ we may combine \eqref{du-o} (restricted to $N$) and \eqref{du-trp} to obtain:
		\begin{eqnarray*}
		Tr(P)=\kappa(\lambda-\rho).
	\end{eqnarray*} In particular, this identity implies that all connected components of $MIN(u)$ have the same dimension, as $Tr(P)$ is constant on $M^n,$ and $\lambda-\rho\leq 0,$ with equality if and only if $R=\frac{n(n-1)\lambda}{m+n-1},$ i.e., when the manifold is Einstein. 
	
We now claim that tangent and normal vector fields to $N$ are eigenvectors of $P,$ corresponding to eigenvalues $\lambda-\rho$ and $0,$ respectively. Indeed, let $p\in N$ and $X\in \mathfrak{X}(N)$ a tangent vector at $p.$ Thereby, as $\nabla u\mid_N =0,$ we have  
\begin{equation*}
		\nabla^2 u(X)(p)=\nabla_{X}\nabla u (p)=0,
	\end{equation*} where we used the fact that this expression depends only on the value of $X(p)$ and the behavior of $\nabla u$ along a curve through $p$ with tangent vector $X.$ Then, from Eq. \eqref{eq-fundamental}, one deduces that
	\begin{equation*}
		0=\nabla_{X}\nabla u (p)=\frac{u}{m}\left(P(X)-(\lambda-\rho)X\right).
	\end{equation*} which yields $P(X)=(\lambda -\rho)X,$ for all $X\in \mathfrak{X}(N).$ Thus, the tangent vectors to $N$ correspond to the eigenvalue $\lambda-\rho.$ Moreover, from Eq. \eqref{l5} in Lemma \ref{lemma-aux}, we know that at $Crit(u)$, 
	\begin{eqnarray*}
		P\circ(P-(\lambda-\rho)I)=0.
	\end{eqnarray*} 
Hence, the only possible eigenvalues of $P$ at $N$ are $\lambda-\rho$ and $0.$ Given that $Tr(P)=\kappa(\lambda-\rho)$ and $\kappa=dim(N),$ we conclude that the normal vectors to $N$ correspond to the eigenvalue $0.$

Therefore, the tensor $P$ restricted to $N$ is represented, in a suitable basis, by the block matrix:

	\begin{eqnarray*}
		P\mid_N=\left(\begin{array}{cc}
		(\lambda-\rho)I_\kappa & 0 \\
			0 & [\textbf{0}]_{n-\kappa}
		\end{array}\right),
	\end{eqnarray*} as $n\times n$ matrix. In terms of the Ricci tensor, this corresponds to

	\begin{eqnarray}\label{eq-trace}
		Ric\mid_N=\left(\begin{array}{cc}
			\lambda I_\kappa & 0 \\
			0 & \frac{(n-1)\lambda- R}{m-1} I_{n-\kappa}
		\end{array}\right).
	\end{eqnarray} Taking the trace in the expression above yields

	\begin{eqnarray*}
		R=\frac{\kappa(m-n)+n(n-1)}{m+n-\kappa-1}\lambda,
	\end{eqnarray*}
	for some $\kappa\in \{0,1\ldots,n-1,n\}.$ This concludes the proof of the theorem.
	
	\end{proof}

	\begin{remark}
	We point out that if $\kappa=0$ or $\kappa=n$ in Theorem \ref{theo1-i}, then one obtains $R=\frac{n(n-1)}{m+n-1}\lambda$ or $R=n\lambda,$ respectively. In both cases, it follows from Eq. (\ref{normricdu}) that $(M^n,\,g)$ is an Einstein manifold. Thus, we may apply Proposition 2.4 of \cite{HPW2} to obtain a classification. In particular, when $R=n\lambda$ and $\mu<0,$ Theorem \ref{theo1-i} yields $dim(Crit(u))=dim(M)=n.$ As $u$ is real analytic in harmonic coordinates, it follows that $u$ must be constant on $M^n.$ 
	\end{remark}

\medskip

\section{Proof of the Main Results}
\label{Sec4}

In this section, we shall present the proofs of Theorems \ref{ThmB} and \ref{ThmA}.

\subsection{Proof of Theorem \ref{ThmB}}

\begin{proof}
	Initially, since $T\equiv 0,$ we may invoke Proposition \ref{propT0} to deduce that either $R=\frac{n(n-1)}{m+n-1}\lambda$ or $R=(n-1)\lambda.$ In the first case, namely when $R=\frac{n(n-1)}{m+n-1}\lambda,$ we may apply Eq. (\ref{normricdu}) to conclude that $|\mathring{Ric}|^2=0,$ and hence $(M^n,\,g)$ is an Einstein manifold. Therefore, by \cite[Proposition 2.4]{HPW2}, the manifold $M^n$ is isometric to one of the following:

\begin{itemize}
\item[(a)] the standard hemisphere $\mathbb{S}_+^n,$ $g=dr^2+\sin^2(r) g_{\mathbb{S}^{n-1}},$ $u=\cos(r),$ $\lambda>0,$ where $r$ is a height function with $r\leq \frac{\pi}{2};$ 
\item[(b)]  $[0,\infty)\times F,$ $g=dt^2+g_{F},$ $u(t,x)=Ct,$ $\lambda=0;$ 
\item[(c)] the hyperbolic space $\mathbb{H}^n,$ $g=dt^2+\sqrt{-k}\sinh^2(\sqrt{-k} t)g_{\mathbb{S}^{n-1}},$ $u=C\cosh (\sqrt{-k}t),$ $\lambda<0;$ 
\item[(d)]  $[0,\infty)\times N,$ $g=dt^2+\sqrt{-k}\cosh^2(\sqrt{-k}t)g_{\mathbb{S}^{n-1}},$ $u(t,x)=C\sinh(\sqrt{-k}t),$ $\lambda<0;$
\item[(e)]  $\mathbb{R}\times F,$ $g=dt^2+e^{2\sqrt{-k}t}g_{F},$ $u(t,x)=Ce^{2\sqrt{-k}t},$ $\lambda<0,$ 
\end{itemize} where $F$ is Ricci flat, $N$ is an Einstein metric with negative Ricci curvature, $C$ is an arbitrary positive constant and $k=\frac{\lambda}{m+n-1}.$

Now, we assume that $R=(n-1)\lambda.$ If $M^n$ is compact without boundary, it follows from \cite{CaseShuWei} that the manifold is trivial. On the other hand, if $M^n$ is compact with boundary, Theorem 2 in \cite{costa-ribeiro-zhou} ensures that $M^n$ is isometric to 

	\begin{itemize}
		\item[(f)] $I\times N$ with product metric, where $N$ is a compact $\lambda$-Einstein manifold.
		\end{itemize}

Next, we suppose that $M^n$ is noncompact. In this case, we already know that $\lambda\leq 0.$ If $\lambda=0,$ Proposition \ref{propL0} implies that  $M^n$ is isometric to 

	\begin{itemize}
		\item[(g)] $[0,\infty)\times F$ with the metric $g=dt^2 +g_{F},$ where $F$ is Ricci flat.
		\end{itemize}

We now assume that $R=(n-1)\lambda$ and $\lambda<0.$ Then, from \eqref{l1}, the eigenvalue $\lambda_1$ of the Ricci tensor associated with the direction $\nabla u$ satisfies $$\lambda_1=|\nabla u|^{-2}Ric(\nabla u,\nabla u)=0.$$ Our next goal is to prove that all remaining eigenvalues of the Ricci tensor are equal to $\lambda.$ First, since $R=(n-1)\lambda,$ we apply Eq. (\ref{normricdu}) to obtain:
	\begin{eqnarray}\label{eq-(n-1)}
		|\mathring{Ric}|^2=\frac{R^2}{n(n-1)}.
	\end{eqnarray}

	Let $\lambda_i,$ $i\geq 2,$ denote the nonzero eigenvalues of the Ricci tensor. Then
	\begin{equation*}
		\sum_{i=2}^{n}(\lambda_i-\lambda)^2 = |Ric|^2-2\lambda R+(n-1)\lambda^2
		=|\mathring{Ric}|^2-\frac{R^2}{n(n-1)},
	\end{equation*} where we used the identity $|\mathring{Ric}|^2=|Ric|^2-\frac{R^2}{n}$ and $R=(n-1)\lambda.$ Substituting \eqref{eq-(n-1)} yields
	
	$$\sum_{i=2}^{n}(\lambda_i-\lambda)^2=0,$$ which implies that $\lambda_2=\ldots=\lambda_{n}=\lambda.$ Therefore, all eigenvalues of the Ricci tensor are constant. Thus, since the Ricci tensor is parallel, one sees from  Eq. \eqref{cotton} that the Cotton tensor also vanishes, and hence, by Lemma \ref{lem1}, we have $W_{ijkl}\nabla_l u=0.$ We are now in a position to invoke \cite[Theorem 1.2]{HPW}, which implies that the metric splits as $$g=dt^2+\varphi^2(t)\widetilde{g}_{_N},$$ where $\widetilde{g}_N$ is $\beta$-Einstein and $u=u(t).$

From Eq. \eqref{l1}, we obtain $$Ric(\nabla u,\nabla u)=\frac{(n-1)\lambda-R}{m-1}(u')^2=0.$$ Hence, applying Proposition \ref{prop-oneil}, one sees that

	\begin{eqnarray*}
		0=Ric(\nabla u,\nabla u)=(u')^2 Ric(\partial t,\partial t)=-(u')^2\frac{(n-1)}{\varphi}\varphi''.
	\end{eqnarray*} Since $u$ is analytical in harmonic coordinates and not constant, it follows that $\varphi''(t)/\varphi(t)=0,$ and thus $\varphi(t)=c$ or $\varphi(t)=c t$, for some positive constant $c.$ However, the latter case is excluded by \cite[Proposition 2]{costa-ribeiro-zhou}. Therefore, $g=dt^2+c^2 \widetilde{g}_{_N}$, where $\widetilde{g}_{_N}$ is $\beta$-Einstein. Since the scalar curvature is negative, we then deduce that
	$$R=\frac{(n-1)\beta}{c^2}=(n-1)\lambda<0,\,\,\,\hbox{i.e.,}\,\,\,\beta=c^2\lambda<0.$$ Thereby, $(N,\,g_N)$ with $g_N=c^2 \widetilde{g}_{N}$ is $\lambda$-Einstein. 
	
	Now, by \cite[Proposition 3.11]{HPW2}, the potential function $u(t)$ of an $m$-quasi-Einstein manifold with $\lambda<0$ must be of the form:
	
	\begin{eqnarray}
		\label{eq-pv}
		\left\{%
		\begin{array}{lll}
			u(t)=\cosh(t), & \hbox{if $\mu<0$,} \\
			u(t)=e^t, & \hbox{if $\mu=0$,} \\
			u(t)=\sinh(t), & \hbox{if $\mu>0$.} \\
		\end{array}%
		\right.
	\end{eqnarray} 
As it follows from the above that the warped product $I\times N$ is rigid, we apply Proposition \ref{prop-hpw} along with \eqref{eq-pv} to conclude that $(M^n,g)$ is isometric to one of the following:
	\begin{itemize}
		\item[(h)] $\mathbb{R}\times N$, with $u(t,x)=\cosh(t)$, when $\mu<0;$
		\item[(i)] $\mathbb{R}\times N$ with $u(t,x)=e^{t}$, when $\mu=0;$
		\item[(j)] $[0,\infty)\times N$ with $u(t,x)=\sinh(t)$, when $\mu>0$,
	\end{itemize} where $N$ is $\lambda$-Einstein.

	The converse statement is straightforward, which completes the proof of Theorem \ref{ThmB}. 
\end{proof}

\subsection{Proof of Theorem \ref{ThmA}}

\begin{proof}
Since $(M^3,g)$ has constant scalar curvature and, by Equation \eqref{l1} in Lemma \ref{lemma-aux}, $\nabla u$ is an eigenvector of 
$P,$ we consider an orthonormal frame $\{e_i\}_{i=1}^{3}$ that diagonalizes $P,$ with $e_1=-\frac{\nabla u}{|\nabla u|}$ and let $\mu_i$ denote the eigenvalues corresponding to $e_i,$ for $i=1,2,3.$ In this coordinate system, it follows that
\begin{eqnarray*}
    P=\left(\begin{array}{ccc}
        0 & 0 & 0\\
        0 & \mu_2 & 0\\
        0 & 0 & \mu_3
    \end{array}\right), 
\end{eqnarray*} where
\begin{eqnarray}\label{system1}
    \begin{cases}
    \mu_2+\mu_3=Tr(P),\\
    \mu_{2}^{2}+\mu_{3}^{2}=|P|^2.
    \end{cases}
\end{eqnarray} Taking into account that $\rho=\frac{(n-1)\lambda-R}{m-1}$, one can rewrite the scalar curvature in dimension three as 
\begin{equation}
\label{p12an}
R=2\lambda-(m-1)\rho.
\end{equation} Then, using Equation \eqref{l2} in Lemma \ref{lemma-aux}, we deduce:
\begin{eqnarray*}
    |P|^2=(\lambda-\rho)Tr(P)=(\lambda-\rho)(2\lambda-(m+2)\rho).
\end{eqnarray*} 
Substituting this into \eqref{system1} yields
\begin{eqnarray*}
\mu_{2}^{2}+\mu_{3}^{2}=(\lambda-\rho)(2\lambda-(m+2)\rho).
\end{eqnarray*} 

Now, observing that $(\mu_2+\mu_3)^2=(Tr(P))^{2},$ we have $\mu_{2}^{2}+\mu_{3}^{2}+2\mu_2 \mu_3=(Tr(P))^{2}.$ Hence, one obtains that
\begin{eqnarray}\nonumber
    2\mu_2 \mu_3&=&(Tr(P))^{2}-(\mu_{2}^{2}+\mu_{3}^{2})\\\nonumber
    &=& (Tr(P))[Tr(P)-(\lambda-\rho)]\\\nonumber
    &=&(Tr(P))[2\lambda-(m+2)\rho-\lambda+\rho]\\\label{eq4}
    &=& [2\lambda-(m+2)\rho][\lambda-(m+1)\rho],
\end{eqnarray} where we have used that $Tr(P)=2\lambda-(m+2)\rho$. This shows that $\mu_2 \mu_3$ is constant. 

We now assume that $\mu_2\mu_3=0.$ By analyticity of the metric $g,$ it follows that either  $\mu_2=0$ or $\mu_3=0.$ Suppose $\mu_3=0,$ then by \eqref{system1}, we obtain that $\mu_2=Tr(P).$ Therefore, all eigenvalues of $P$ are constant, and hence $P$ is parallel. In particular, the Cotton tensor $C_{ijk}$ vanishes, and since the Weyl tensor is identically zero in dimension three, it follows that $T\equiv 0$ (see Lemma \ref{lem1}). Furthermore, Proposition \ref{prop1-p} implies that  $\mu_{2}=\mu_3=0,$ so $P=0$ and $M^3$ is Einstein. Thus, by \cite[Proposition 2.4]{HPW2}, we conclude that $M^3$ is isometric to either $\mathbb{H}^{3}$, or $\mathbb{R}\times_{e^r}\mathbb{R}^2$, or $[0,\infty)\times_{\cosh(r)}\mathbb{H}^2.$

On the other hand, suppose $\mu_2 \mu_3\neq 0.$ Then,
\begin{eqnarray*}
    \mu_2=\frac{1}{{2\mu_3}}[2\lambda-(m+2)\rho][\lambda-(m+1)\rho]=\frac{\alpha}{2\mu_3},   
\end{eqnarray*} where $\alpha=(2\lambda-(m+2)\rho)(\lambda-(m+1)\rho)$.
Using \eqref{system1}, we have $$Tr(P)=\mu_2+\mu_3=\frac{\alpha}{2\mu_3}+\mu_3.$$ Multiplying both sides by $2\mu_{3},$ we get the polynomial:
\begin{eqnarray}
\label{poly}
    2\mu_{3}^{2}+\alpha=2\mu_3(2\lambda-(m+2)\rho).   
\end{eqnarray}
Computing the discriminant for $\mu_3,$ we infer

\begin{eqnarray*}
    \Delta&=&4(2\lambda-(m+2)\rho)^2-8[2\lambda-(m+2)\rho][\lambda-(m+1)\rho]\\
    &=&4(2\lambda-(m+2)\rho)[(2\lambda-(m+2)\rho)-2\lambda+2(m+1)\rho]\\
    &=&4m\rho(2\lambda-(m+2)\rho)^2.
\end{eqnarray*} Therefore, the roots of \eqref{poly} are:

\begin{eqnarray*}
    \mu_3= \frac{(2\lambda-(m+2)\rho)\pm (2\lambda-(m+2)\rho)\sqrt{m\rho}}{2}.
\end{eqnarray*} Moreover, notice that the eigenvalue $\mu_2$ satisfies an expression equivalent to \eqref{poly} and, by \eqref{system1}, one sees that
\begin{eqnarray*}
    \mu_2= \frac{(2\lambda-(m+2)\rho)\mp (2\lambda-(m+2)\rho)\sqrt{m\rho}}{2}.
\end{eqnarray*} Hence, both $\mu_2$ and $\mu_3$ are constants. However, as stated in Corollary 4.1 of \cite{HPW2}, we have $\rho \leq 0.$ At the same time, since the eigenvalues are real, we consider $\Delta \geq 0.$ So, it follows that either $\rho = 0$ or $\rho = \frac{2\lambda}{m+2}$. In the latter case, the roots of \eqref{poly} would imply that both $\mu_2$ and $\mu_3$ vanish, which contradicts the fact that $\mu_2 \mu_3\neq 0.$ Therefore, we must have $\rho = 0,$ and from Eq. (\ref{p12an}), one sees that $R = 2\lambda.$ Then, by Proposition \ref{propT0}, we conclude that $T = 0.$ Thus, it suffices to apply Theorem \ref{ThmB} to deduce that $(M^3, g)$ is isometric, up to scaling, to either $\mathbb{R} \times \mathbb{H}^2$ or $[0, \infty) \times \mathbb{H}^2$.

Finally, if $\lambda=0,$ we invoke Proposition \ref{propL0} to conclude that $M^2$ is isometric to $[0,\infty)\times \mathbb{R}^2$ with the product metric $g=dt^2+g_{\mathbb{R}^2}.$  

This completes the proof of Theorem \ref{ThmA}.

\end{proof}

\end{document}